\documentclass[sn-basic]{sn-jnl}

\jyear{2021}%
\usepackage{hyperref}

\theoremstyle{thmstyleone}%
\newtheorem{theorem}{Theorem}
\newtheorem{lemma}{Lemma}
\newtheorem{corollary}{Corollary}
\newtheorem{proposition}[theorem]{Proposition}%

\theoremstyle{thmstyletwo}%
\newtheorem{remark}{Remark}%
\newtheorem{notation}{Notation}

\theoremstyle{thmstylethree}%
\newtheorem{definition}{Definition}%

\begin{document}

\title[Convolutional Persistence Transforms]{Convolutional Persistence Transforms}

\author*[1,3]{\fnm{Yitzchak Elchanan} \sur{Solomon}}\email{elchanansolomon@gmail.com}

\author[1,2,3]{\fnm{Paul} \sur{Bendich}}\email{paul.bendich@duke.edu}

\affil*[1]{\orgdiv{Mathematics}, \orgname{Duke University}, \orgaddress{\city{Durham}, \state{NC}, \country{USA}}}

\affil[2]{\orgname{Geometric Data Analytics}, \orgaddress{\city{Durham}, \state{NC}, \country{USA}}}

\affil[3]{\small The authors were partially supported by the Air Force Office of Scientific Research under the  grant ``Geometry and Topology for Data Analysis and Fusion",  AFOSR FA9550-18-1-0266. They would also like to extend their thanks to Alexander Wagner, for helpful conversations.}


\abstract{ In this paper, we consider topological featurizations of data defined over simplicial complexes, like images and labeled graphs, obtained by convolving this data with various filters before computing persistence. Viewing a convolution filter as a local motif, the persistence diagram of the resulting convolution describes the way the motif is distributed across the simplicial complex. This pipeline, which we call \emph{convolutional persistence}, extends the capacity of topology to observe patterns in such data. Moreover, we prove that (generically speaking) for any two labeled complexes one can find \emph{some} filter for which they produce different persistence diagrams, so that the collection of all possible convolutional persistence diagrams is an injective invariant. This is proven by showing convolutional persistence to be a special case of another topological invariant, the Persistent Homology Transform. Other advantages of convolutional persistence are improved stability, greater flexibility for data-dependent vectorizations, and reduced computational complexity for certain data types.	Additionally, we have a suite of experiments showing that convolutions greatly improve the predictive power of persistence on a host of classification tasks, even if one uses random filters and vectorizes the resulting diagrams by recording only their total persistences.}



\maketitle

\section{Overview}
Persistent homology is a method of assigning multiscale topological descriptors to parametric families of shapes. In \emph{functional} persistence, the object of study is a real-valued function $f:X \to \mathbb{R}$ defined over a space $X$, and the parametric family of shapes are the sublevel-sets $X_{\alpha} = \{x \in X : f(x) \leq \alpha\}$. It is similarly possible to consider superlevel-sets, which is equivalent to negating the \emph{filter function} $f$. One crucial feature of this construction is that $X_{\alpha}$ is a subset of $X_{\beta}$ for $\alpha \leq \beta$, so that the sublevel-sets are naturally nested. The output of persistent homology is a collection of intervals (equivalently, a collection of points), called a \emph{barcode} (or \emph{persistence diagram}, using the point representation). The space of barcodes is not a vector space, even approximately \citep{bubenik2020embeddings,wagner2021nonembeddability}, but there do exist multiple \emph{vectorizations} \citep{adams2017persistence,bubenik2015statistical,monod2019tropical,di2015comparing,carriere2015stable} that transform barcodes into vectors suitable for machine learning and data analysis.\\

A very general setting for functional persistence is that of \emph{simplicial complexes}, i.e. shapes obtained by gluing together points, edges, triangles, tetrahedra, etc. For example, if $X$ is a triangulation of a 2D rectangular grid, a function $f:X \to [0,1]$ can be viewed as a greyscale image. Persistent homology can then be understood as a feature extraction method for such images, either for supervised or unsupervised learning. Example applications include removing image noise \citep{chung2018topological,chung2022multi}, parameter estimation for PDEs \citep{calcina2021parameter,adams2017persistence}, segmentation \citep{hu2019topology}, flow estimation \citep{suzuki2021flow}, tumor analysis \citep{crawford2020predicting}, cell immune micronenvironment \citep{aukerman2020persistent}, and materials science \citep{hiraoka2016hierarchical}.\\

It should be understood that purely topological methods do not provide state-of-the-art predictive accuracy on most machine learning tasks, and are not to be considered as \emph{alternatives} to more general methods like neural nets and kernel methods. Rather, topological methods provide \emph{principled} and \emph{intepretable} features that are different from those extracted by other methods, and can help improve the performance and utility of the entire pipeline, especially in settings where data is limited, see \citep{khramtsova2022rethinking}.\\

To that end, it is important to understand the properties of functional persistence and their role in machine learning. Here we consider some of the most salient properties:
\begin{itemize}
	\item (Computational complexity) For a simplicial complex $K$ with $N$ simplices, computing persistence is $O(N^\omega)$, where $\omega$ is the matrix multiplication constant \citep{milosavljevic2011zigzag}. This means that such persistence calculations scale poorly in the resolution of data, especially high-dimensional data, where doubling the resolution in $\mathbb{R}^d$ results in a $2^{d}$-fold increase in the number of simplices.  Still, in many settings, persistence calculations are on average much faster than the pessimistic worst-case complexity, see e.g. \citep{giunti2021average}. 
	\item (Stability) Persistent homology is stable to small perturbations of the input data, in that the distance between the barcodes for two functions $f$ and $g$ on a common space $X$ is bounded by $\|f-g\|_{\infty}$. However, persistence is not stable to outliers, so images that look similar outside of a small fraction of pixels can produce wildly different barcodes.
	\item (Flexibility) There is a single persistence diagram to be associated with each pair $(X,f)$ of space $X$ and real-valued function $f$. This lack of additional parameters makes applying persistence straightforward, but can also be limiting in contexts where data-dependent featurizations are desired.
	\item (Invertibility) There exist many distinct space-function pairs $(X,f)$ producing identical barcodes. Thus, persistence is not invertible as feature map, and this loss of information may hinder the capacity of persistence-based methods to identify patterns or distinguish distinct images.
\end{itemize}

Our goal in this paper is to introduce a modification to persistence of certain simplicial complexes that is (1) often faster to compute, (2) more robust to outliers, (3) allows for data-driven tuning, (4) is provably invertible, and (5) is more informative in practice. Essentially, this technique consists of passing multiple convolutional filters over the data before computing persistence, and so we call it \emph{convolutional persistence}. A crucial insight in this paper is that convolutional persistence can generically be transformed into a special case of the Persistent Homology Transform (PHT) \citep{turner2014persistent} in high-dimensions, implying that it shares the theoretical properties enjoyed by that invariant.\\

To get an intuitive sense for why convolutions might be of value in applied topology, consider Figure \ref{fig:crosses}, which shows how convolution with a simple filter makes geometric information in an image more available to topological methods.\\

\begin{figure}
	\centering
	\includegraphics[scale=0.4]{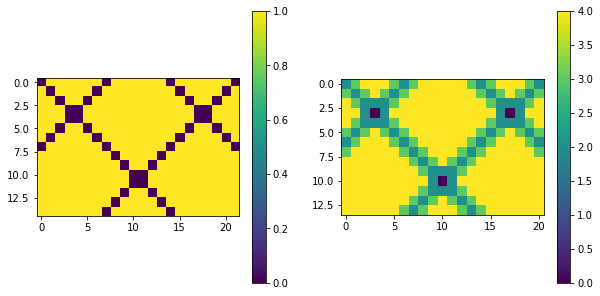}
	\caption{The original image is shown on the left. The image has interesting topological structure but the sublevel-set persistence is trivial. On the right, we show the result of convolving with a $2 \times 2$-filter with a value of $1$ in each pixel. The sublevel-set persistence of the convolved image is more informative; at height zero, there are three connected components, corresponding to three regions in the image where the convolution is zero. These regions then merge at an immediate value, due to those parts of the image where the local patches are more similar to the filter.}
	\label{fig:crosses}
\end{figure}

\subsection{Results}

The major theoretical result of the paper is that convolutional persistence is generically injective (Theorems \ref{thm:simpphtinj} and \ref{thm:thmcptim}). We also provide stability results for convolutional persistence (Propositions \ref{prop:stability} and \ref{prop:simpstability}). The experimental results of Section \ref{sec:experiments} show that convolutional persistence significantly outperforms ordinary persistence at a collection of classification tasks. Interestingly enough, random filters perform well for the classification tasks considered, even when only the total persistences of the resulting diagrams are taken as feature vectors.

\subsection{Organization}

The remainder of the paper is organized as follows. Section \ref{sec:background} provides a thorough, non-technical survey on persistent homology. Section \ref{sec:convpers} defines two versions of \emph{convolutional persistence} and reviews prior related work. Section \ref{sec:theoreticalresults} contains the main theoretical results of the paper. Section \ref{sec:experiments} compares ordinary and convolutional persistence on a host of image datasets, showing the capacity of convolutional persistence to produce features well-suited for image classification. Finally, Section \ref{sec:conclusion} discusses outstanding questions and potential generalizations of this work. Code for running experiments with convolutional persistence can be found at \href{https://github.com/yesolomon/convpers}{https://github.com/yesolomon/convpers}.

\section{Background}
\label{sec:background}
The content of this paper assumes familiarity with the concepts and tools of persistent homology. Interested readers can consult the articles of Carlsson \citep{carlsson2009topology} and Ghrist \citep{ghrist2008barcodes} and the textbooks of 
Edelsbrunner and Harer \citep{edelsbrunner2010computational} and Oudot \citep{oudot2015persistence}. We include the following primer for readers interested in a high-level, non-technical summary.

\subsection{Persistent Homology}

Persistent homology records the way topology evolves in a parametrized sequence of spaces. In the case of functional persistence, we consider an ambient space $X$ equipped with a real-valued function $f:X \to \mathbb{R}$. The sublevel-sets $X_{\alpha} = \{x \in X : f(x) \leq \alpha\}$ of $f$ are naturally nested, in that $X_{\alpha}$ is a subset of $X_{\beta}$ for $\alpha \leq \beta$, and these form a \emph{filtration} of $X$. Persistence records how the topology of this filtration evolves as a function of the parameter $\alpha$.\\

Simplicial complexes admit particularly simple filtrations. To every simplex $\sigma$ we can associate a real value $f(\sigma)$ that encodes the parameter value at which it appears in the filtration of $X$. The only restriction on the function $f$ is the following consistency condition: if $\sigma$ is a sub-simplex of a higher-dimensional simplex $\tau$ (i.e. an edge which sits at the boundary of a triangle), we must have $f(\sigma) \leq f(\tau)$, ensuring that simplices do not appear before any of their faces.\\

Given a filtered simplicial complex, as the sequence of sublevel-sets evolves, the addition of certain edges or higher-dimensional simplices alters the \emph{topological type} of the space. A precise way of quantifying topology is \emph{homology}, which measures the number of connected components (zero-dimensional homology), cycles (one-dimensional homology), or voids (higher-dimensional homology) in a space. Thus, homology can change when two connected components merge or a new cycle is formed. Simplices responsible for such topological changes are called \emph{critical}. Persistent homology records the parameter values at which critical simplices appear, notes the dimension in which the homology changes, and pairs critical values by matching the critical value at which a new homological feature appears to the critical value at which it disappears. This information is then organized into a structure called a \emph{barcode}, which is simply a collection of intervals. Figure \ref{fig:cubpers} shows the computation of the zero- and one-dimensional barcodes for a simple simplicial complex.\\

A barcode can also be encoded as a collection of points, simply by mapping each interval $[a,b]$ to the point $(a,b) \in \{(a,b) \mid a<b, \,\, a,b \in \mathbb{R} \cup \infty\}$. This collection of points is called a \emph{persistence diagram}, and besides for certain subtleties, such as dealing with open or closed endpoints and  keeping track of points with multiplicity, it contains the exact same information as the original barcode. However, for certain analyses and vectorizations of persistence features, persistence diagrams are more useful than barcodes, as will be shown below.\\

\begin{figure}
\centering
\includegraphics[scale=0.5]{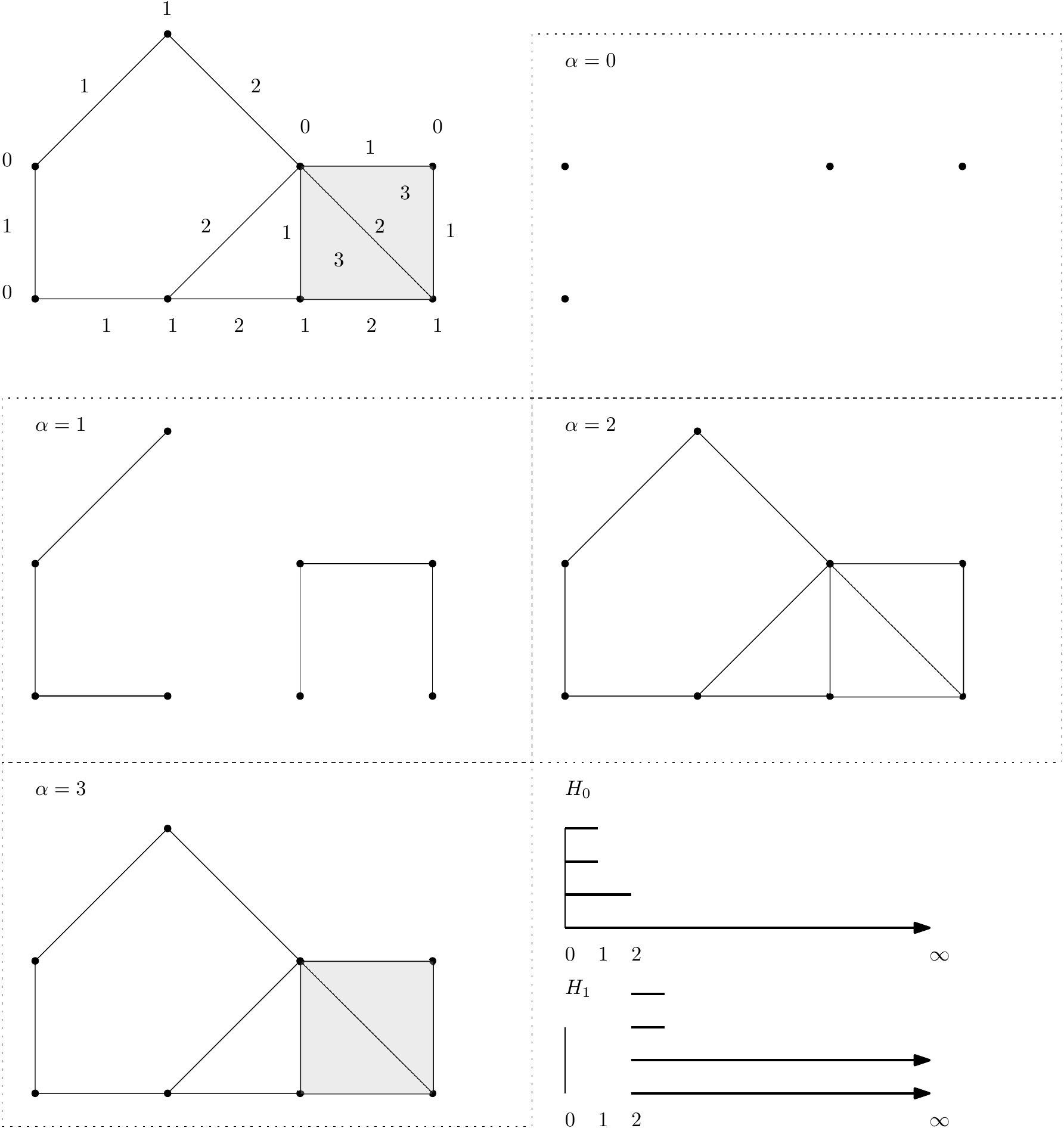}
\caption{Top-left: A simplicial complex with filtration values attached to vertices, edges, and squares. Top-right through bottom-left: sublevel-sets associated with different threshold values. Bottom-right: barcodes in dimensions zero and one. The zero-dimensional homology $H_0$ barcode contains four bars, since at $\alpha=0$ there are four connected components. Two bars die at $\alpha=1$, since at that threshold value there are only two connected components. Finally, $\alpha=2$ sees the merger of these connected component, so another bar dies at $\alpha=2$ and the last persists to infinity. In one-dimensional homology $H_1$, three bars are born at $\alpha = 2$, when three loops appear in the sublevel-set, and one of these bars dies at $\alpha=3$, when that loop is killed off by the introduction of a square.}
\label{fig:cubpers}
\end{figure}

When working with simplicial complexes, it is possible to limit the maximal dimension of simplices allowed in the construction. Given an arbitrary simplicial complex $K$, we write $K^m$ to denote the subcomplex consisting of all simplices of dimension at most $m$; this is called the \emph{$m$-skeleton of $K$}. Note that $K$ and $K^m$ have the same homology in dimensions less than $m$, but may differ in dimension $m$, and $K^m$ has no homology in dimension greater than $m$.

\subsection{Image Cubical Complexes}
\label{sec:imagecomplex}

A common alternative to simplicial complexes are \emph{cubical complexes}, i.e. shapes that are built from vertices, edges, squares, cubes, etc. Such complexes are more natural when working with data built on top of a grid, such as pixel or voxel data, see \citep{kaczynski2004computational} for details. A cubical complex can always be converted into a simplicial complex by cutting up cubes into simplicial pieces: a canonical way of doing this is the \emph{Freudenthal triangulation}. For a square, for example, this simply amounts to inserting a diagonal edge that produces two triangles. Given a function on a cubical complex, there is a canonical way to extend this function to its Freudenthal triangulation without changing its persistence; in the square example, in which the Freudenthal triangulation incorporates a diagonal edge and two triangles, we simply assign to these three simplices the filtration value assigned to the original square in the original cubical complex. Thus, from the perspective of persistence, it does not fundamentally matter if we work with cubical or simplicial complexes, although there is a difference vis-a-vis computations.\\ 

Given a $d$-dimensional grayscale image, there are two ways of turning this data into a cubical complex. One is to view the voxels as being vertices, and higher-dimensional cubes as coming from voxels adjacencies, so that pairs of adjacent pixels form an edge and squares come from four voxels in a square formation, etc. There is a canonical way of extending the function $f$ from the voxels (vertices) to the entire complex: given a cube $\tau$, define $f(\tau)$ to be $\max_{\sigma} f(\sigma)$, where the max is taken over all vertices $\sigma < \tau$. Thus, a square appears precisely once all its constituent vertices appear; this is called the \emph{lower-*} filtration.\\

Alternatively, one can also view the voxels as being $d$-dimensional cubes, and have the lower-dimensional cubes be the faces of these voxels. As before, there is a canonical way of extending the function value from the voxels (top-dimensional cubes) down to entire complex: for a cube $\sigma$, define $f(\sigma)$ to be $\min_{\tau} f(\tau)$, where the min is taken over all voxels $\tau$ that contain $\sigma$. Thus, a cube appears precisely when at least one of the voxels in which it participates does;  this is called the \emph{upper-*} filtration.\\ 

Consult figure \ref{fig:image_to_complex} for an illustration of these two images complexes. Generally, these complexes will differ, and the resulting persistence barcodes will be different. However, there exists a formula for reading the barcodes for one construction from the barcodes of the other, see \citep{bleile2021persistent}.

\begin{figure}
\centering
\includegraphics[scale=2]{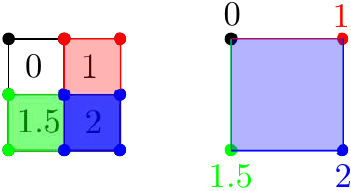}
\caption{A $2 \times 2$ image can be turned into a complex in one of two ways. Left: A complex with four top-dimensional cubes, with function values, indicated using color, extended to vertices and edges via the upper-* rule. Right: A complex with four vertices, with function values extended to the edges and interior square via the lower-* rule.}
\label{fig:image_to_complex}
\end{figure}
\subsection{Comparing Persistence Diagrams}

As multi-sets of points in the plane, persistence diagrams are not vectors. However, there exist multiple metrics for comparing persistence diagrams. The most common approach is to view persistence diagrams as discrete distributions on the plane $\mathbb{R}^2$, and use techniques from optimal transport theory, such as $p$-Wasserstein metrics $W_p$, to compare them, see \citep{oudot2017persistence} Chapter 3 and \citep{villani2021topics}. This analogy is complicated by the fact that persistence diagrams do not all have the same number of points, and that points in a persistence diagram near the diagonal line $y=x$ correspond to transient homological features, dying shortly after they are born, which ought not to play an important role in dictating similarity of diagrams. These problems are ameliorated by modifying the optimal transport protocol to allow paring points in one diagram either with points in the other diagram or with the diagonal line $y=x$, the latter incurring a cost proportional to the distance of the given point from the diagonal. This \emph{diagonal paring} allows us to define an optimal transport distance between any pair of diagrams, regardless of how many points they have. In practice, the $p$-Wasserstein metrics in use are either $p=1,p=2$, or $p=\infty$, the latter of which is called the \emph{Bottleneck distance} in persistence theory, and is written $d_B$. See Figure \ref{fig:wasserstein} for a visualization of a matching between persistence diagrams.  

\begin{figure}
\centering
\includegraphics[scale=1]{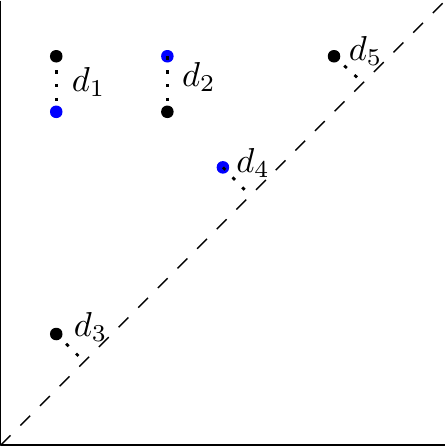}
\caption{Two persistence diagram $D_1$ and $D_2$, one in black and the other in blue. An optimal matching between these diagrams is shown, containing both diagonal and non-diagonal pairings. The resulting $p$-Wasserstein distance is then $W_{p}(D_1,D_2) = \sqrt[p]{d_1^p + d_2^p + d_3^p + d_4^p + d_5^p}$.}
\label{fig:wasserstein}
\end{figure}

\subsection{Computational Complexity}
Persistence calculations are $O(N^{\omega})$, where $N$ is the number of simplices in the filtration and $\omega$ is the matrix multiplication constant \citep{milosavljevic2011zigzag}. Though this bound is in practice quite pessimistic, it is accurate in reflecting the poor scaling of persistence calculations in the resolution of an image \citep{otter2017roadmap}.\\

The bottleneck distance between persistence diagrams can be computed exactly using the well-known Hungarian algorithm, whose complexity is $O(n^3)$ for pairs of diagrams with at most $n$ points. The number of points in the persistence diagram of filtration is always bounded by the total number $N$ of simplices in the filtration, and may be much smaller. If one is interested in \emph{approximating} the bottleneck, or more generally, $p$-Wasserstein distances between persistence diagrams, there are significantly faster methods than the Hungarian algorithm. One such method is based on \emph{entropic approximation} techniques from optimal transport \citep{lacombe2018large}, and provided that the norms of the points in the diagrams are uniformly bounded by some constant $C$, a transport plan within $\epsilon$ of the optimal matching can be produced in $O(\frac{n \log(n) C}{\epsilon^2})$ iterations of the Sinkhorn algorithm, which is itself linear in $n$.  Moreover, recent work in optimal transport theory provide an improved asymptotic convergence rates for Sinkhorn's algorithm of $\tilde{O}(n^2/\epsilon)$, where $\tilde{O}$ indicates that logarithmic factors have been hidden, see \citep{pham2020unbalanced}.


\subsection{Properties of Persistent Homology}
\label{subsec:persprop}
Persistence theory guarantees that a small modification to the filtration of a space produces only small changes in its persistence diagram. 

\begin{theorem}[\citep{cohen2007stability}\footnote{The main theorem of \citep{cohen2007stability} is much more general, applying to tame functions on triangulable spaces, conditions which are automatically satisfied in the simplicial setting.}]
\label{thm:perstab}
Let $f,g:X \to \mathbb{R}$ be two filtrations on a simplicial complex. Then the Bottleneck distance between their persistence diagrams is bounded by $\|f-g\|_{\infty}$.
\end{theorem}

\begin{corollary}
	\label{cor:infstabcor}
	Let $f,g: K^{0} \to \mathbb{R}$ be two functions on the vertex set of a simplicial complex, with lower-* extensions $\hat{f}$ and $\hat{g}$, respectively. Then $\|\hat{f}-\hat{g}\|_{\infty} \leq \|f-g\|_{\infty}$, so that the Bottleneck distance between the persistence diagrams of $\hat{f}$ and $\hat{g}$ is bounded by $\|f-g\|_{\infty}$.
\end{corollary}
\begin{proof}
Consider a simplex $\sigma \in K$. By definition, $\hat{f}(\sigma) = \max_{v \in \sigma^{0}}f(v)$ and $\hat{g}(\sigma) = \max_{v \in \sigma^{0}}g(v)$. Let $v_{f}$ and $v_{g}$ be the vertices of $\sigma$ realizing these maxima for $f$ and $g$, respectively. Then $f(v_{f}) \leq g(v_{f}) + \|f-g\|_{\infty} \leq g(v_{g}) + \|f-g\|_{\infty}$, so that $f(\sigma) \leq g(\sigma) + \|f-g\|_{\infty}$. A symmetric argument shows that $g(\sigma) \leq f(\sigma) + \|f-g\|_{\infty}$, so that $\vert f(\sigma) - g(\sigma)\vert \leq \|f-g\|_{\infty}$.
\end{proof}

This implies that a small error in the filter function produces only small distortion in the resulting persistence diagram. However, persistent homology is not at all stable to \emph{outliers}, i.e. a small subset of the data having large error \citep{buchet2014topological}.\\

More recent work \citep{skraba2020wasserstein} has established analogous stability results for $W_p$ metrics between persistence diagrams, \emph{although we caution readers that this paper has yet be to published}. The latter of the following two results is particularly useful for the subject matter of this paper.

\begin{definition}[\citep{skraba2020wasserstein}, Definition 4.1]
	The $L^p$ norm of a function on a simplicial complex $f:K \to \mathbb{R}$ is:
	\[\|f\|_{p}^{p} = \sum_{\sigma \in K}\vert f(\sigma)\vert^{p}.\]
\end{definition}

\begin{theorem}[\citep{skraba2020wasserstein}, Theorem 4.8]
Let $f,g : K \to \mathbb{R}$ be two monotone functions on a fixed simplicial complex. Then for all $p \geq 1$, \[W_{p}(\operatorname{Diag}(f),\operatorname{Diag}(g)) \leq \|f-g\|_{p}.\]
\end{theorem}

\begin{theorem}[\citep{skraba2020wasserstein}, Theorem 5.1]
Let $f,g:P \to \mathbb{R}$ be the grayscale functions for two images defined over the same grid of pixels $P$. Let $\hat{f}$ and $\hat{g}$ be the corresponding extensions to the $d$-dimensional cubical complex built on top of $P$, by either method described in Section \ref{sec:imagecomplex}. Then we have the stability result:
\[W_{p}(\operatorname{Diag}(\hat{f}),\operatorname{Diag}(\hat{g})) \leq \left( \sum_{i=1}^{d} 2^{d-i} {d \choose i }\right) \|f-g\|_{p} = (3^{d} - 2^{d})\|f-g\|_{p}.\]
\label{thm:STwasstab}
\end{theorem}

Another important feature of persistent homology is that the pipeline mapping a function $f$ on a simplicial complex $K$, thought of as a vector in $\mathbb{R}^{K}$, to its associated persistence diagram is  \emph{almost everywhere differentiable}  \citep{gameiro2016continuation,poulenard2018topological}. This is because the coordinates of the points in the diagram correspond to the values of $f$ on critical simplices, and this correspondence is generically locally stable. However, large perturbations of $f$ will change the location of critical simplices, as well as their pairings \citep{bendich2020stabilizing}, and non-differentiable nonlinearities can be found in those degenerate cases when the critical simplex data is not unique and locally unstable. This almost everywhere differentiability of persistent homology allows us to use gradient descent methods for topological optimization \citep{carriere2021optimizing}.  Applications of topological optimization to machine learning models, and especially to deep learning, were considered in \citep{gabrielsson2020topology}.  \\ 

It is also important to note that persistent homology is not an injective invariant, meaning that different functions on the same space, or on different spaces, can have the same persistence diagrams \citep{curry2018fiber,leygonie2022fiber,leygonie2021algorithmic}, see Figure \ref{fig:samepers}. Thus, some information is lost in the process of converting a pair $(X,f)$  to a persistence diagram.\\

\begin{figure}
\centering
\includegraphics{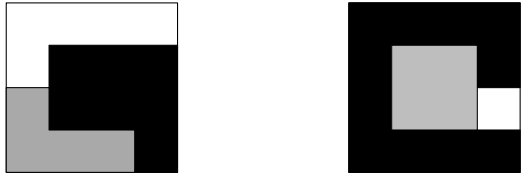}
\caption{Two functions on the same grid with the same persistence, using either sublevel- or superlevel-set filtrations. Color scheme: \{black: 0, grey: 0.5, white: 1\}. All persistence diagrams consist of the single point $(0,1)$ in dimension $0$.}
\label{fig:samepers}
\end{figure}

We conclude the background section with an important lemma about the persistence of simplicial complexes embedded in Euclidean spaces. This lemma equates the persistence of two filtrations on such complexes, both extended from a function on the vertex set, one discretely (via the lower-* filtration) and the other continuously (via linear interpolation).

\begin{definition}
	 Let $K$ be a simplicial complex, and $f: K^{0} \to \mathbb{R}$ a function on the vertices of $K$. We write $\hat{f}$ to denote the lower-* extension of $f$ to all of $K$, defined by $\hat{f}(\sigma) = \max_{v \in \sigma^{0}} f(v)$ for $\sigma$ a simplex with vertex set $\sigma^{0}$. 
\end{definition}

\begin{proposition}
Let $K$ be an embedded simplicial complex in $\mathbb{R}^d$ with vertex set $K^{0}$. For a function $f:K^{0} \to \mathbb{R}$ with lower-* extension $\hat{f}$, we can also extend $f$ to all of $K$ via linear interpolation, producing a function $f_{*}:K \to \mathbb{R}$. Then these two extensions produce identical persistence diagrams, i.e. $\operatorname{Diag}(K,f_{*}) = \operatorname{Diag}(K,\hat{f})$. 
\label{prop:twofilt}
\end{proposition}

\begin{proof}
This follows from Lemma 2.3 of \citep{bestvina1997morse} and Section VI.3 of \citep{edelsbrunner2010computational}.
\end{proof}

\section{Convolutional Persistence}
\label{sec:convpers}

 We now introduce two modifications to functional persistence of complexes, one based on techniques in \emph{image processing} and applying to cubical complexes, and another based on \emph{graph convolutions} and applying to either cubical or simplicial complexes (although, without loss of generality, we focus on the latter -- see Section \ref{sec:imagecomplex} for how a cubical complex can be converted into a simplicial complex). We call the former \emph{Image Convolutional Persistence} and the latter \emph{Simplicial Convolutional Persistence}.

\subsection{Image Convolutional Persistence}

Let $P \subset \mathbb{Z}^d$ be a rectangle inside of the integer lattice, and let $f: P \to \mathbb{R}$ be a function defined on $P$. $P$ can be viewed as the vertex set of an $m$-dimensional cubical complex $K^{m}_{P}$, and $f$ can be extended to this complex via the lower-* rule, in which $f(\sigma) = \max_{p \in \sigma}f(p)$, as in Section \ref{sec:imagecomplex}.\\

\begin{definition}
Let $B \subset \mathbb{Z}^d$ be another rectangle, and let $g: B \to \mathbb{R}$ be a function on this rectangle; the pair $(B,g)$ acts as a \emph{convolutional filter}. Fix a vector $k = (k_1, \cdots, k_d)$ with $k_i \in \mathbb{N}_{>0}$, corresponding to the \emph{stride} of the convolution. For $v \in \mathbb{Z}^d$ with $B + (v \odot k) \subseteq P$,\footnote{The symbol $\odot$ refers to the Hadamard product, which performs componentwise multiplication between vectors.} define:
\[(f \ast g)(v) = \sum_{p \in B} g(p)f(p+v \odot k).\]
Let $R \subset \mathbb{Z}^d$ be the collection of values $r$ such that $B + (r \odot k) \subseteq P$, which is necessarily also a rectangle. The pair $(R, f \ast g)$ is the output of our convolution. See Figure \ref{fig:convolution}.  We will generally assume that $R$ is not empty, meaning that \emph{some} translate of $B$ is a subset of $P$.
\end{definition}

\begin{remark}
	 If $B$ consists of a single element $b$, and $g(b) = 1$, then, up to translation, $P=R$ and $f = f \ast g$.
\end{remark}

\begin{remark}
	 Translating $P$ or $B$ has no substantial impact on the above construction, resulting only in a translation of the corresponding region $R$. Since all of the invariants we study here are invariant to translation of the domain, it would be possible to work entirely in the quotient space of grid functions $f:P \to \mathbb{R}$ \emph{defined up to translation}. Equivalently, we might work with rectangles $P$ whose lattice-minimal element (i.e. bottom-left corner) is the zero vector. Ultimately, we have chosen to eschew these additional formalisms in the interest of keeping our definitions as concrete as possible, but encourage the reader to adopt whichever perspective they find most insightful.  
\end{remark}

\begin{definition}
As a generalization of the above definition, we can allow $n$-channel images $f:P \to \mathbb{R}^n$ and filters $g:B \to \mathbb{R}^n$, and replace the product in the formula above with a dot product:
\[(f \ast g)(v) = \sum_{p \in B} g(p)\cdot f(p+v \odot k).\]
Thus, the result of the convolution is real-valued (not vector valued), and ordinary persistent homology can be computed. In principle, it is possible to have the image and filter dimensions differ, and to compute the persistent of multi-channel images via \emph{multiparameter persistence}, see \citep{botnan2022introduction}, but this is beyond the scope of this work and presents its own set of theoretical and computational challenges.
\end{definition}

\begin{figure}
\centering
\includegraphics[scale=0.7]{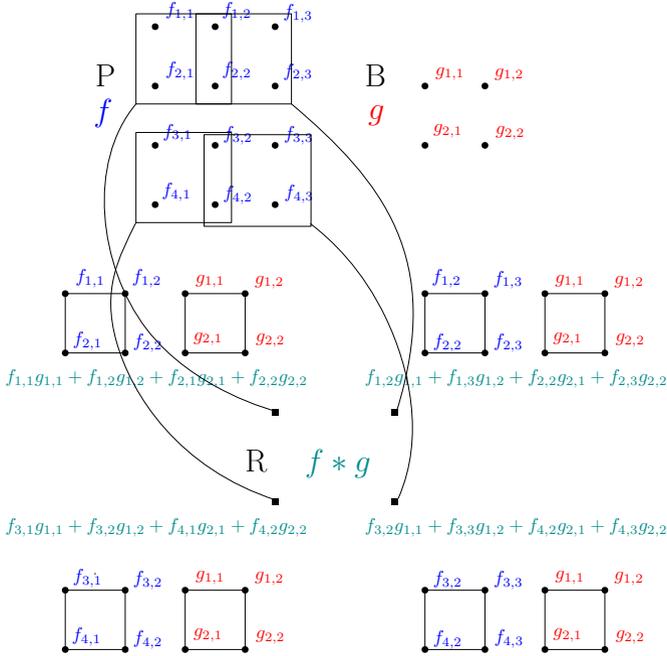}
\caption{$f$ is a function on a $4 \times 3$ grid $P$, with function values indicated by subscripts. $g$ is a filter function on a $2 \times 2$ grid $B$. Using a stride vector $k = (1,2)$ means that the rectangle $B$ is translated horizontally one unit but vertically two units; there are four possible translates of $B$ that sit inside $P$, corresponding to the $2 \times 2$ grid $R$. Each vertex of $R$ corresponds to a translate of $B$, and the value of $f \ast g$ at that vertex is obtained by lining up the values of $g$ and $f$ on that translate, multiplying them element-wise, and then summing them.}
\label{fig:convolution}
\end{figure}

$R$ may be much smaller than the support of $f$, since the various translates of $B$ covering this support are not required to overlap much, if at all. In the degenerate setting where $P$ and $B$ are the same shape, $R$ will consist of a single vertex. For the purposes of computing persistence, we will think of $R$, like $P$, as being the vertex set of a cubical complex $K_{R}^{m}$, and we extend functions on $R$ to the entirety of the cubical complex using the lower-* rule.

\begin{notation}
For a function $f$ defined on the vertex set of a complex $K$, we write $\hat{f}$ to denote its lower-* extension to all of $K$.
\end{notation}

We propose that the collection of persistence diagrams of the form $PH(K_{R}^{m},\widehat{f \ast g_{i}})$, for some set of filter functions $\mathcal{G} = \{g_i\}$, is of greater general utility than $PH(K_{P}^{m},\hat{f})$. The computational advantages are immediate: when the stride $k$ is large, so that $R$ is much smaller than $P$, we have replaced a single, very expensive calculation with multiple, significantly faster calculations that can be performed in parallel. Moreover, we claim that this approach, which we deem \emph{convolutional persistence} has superior inverse properties. Proofs of stability and injectivity can be found in the following Section \ref{sec:theoreticalresults}.\\

The flexibility of convolutional persistence comes from allowing the collection of filters $\mathcal{G}$ to be curated for the task at hand. Indeed, there are many ways for this choice to be made:
\begin{enumerate}
\item Take $\mathcal{G}$ to consist of a collection of popular filters in image processing, like blurring, sharpening, and boundary detection.
\item Take $\mathcal{G}$ to consist of eigenfilters identified via PCA on the set of patches of images in the training set.
\item Picking filters \emph{at random}.
\item Incorporate convolutional persistence in a deep learning pipeline and learn $\mathcal{G}$ to minimize a chosen loss function.
\end{enumerate}

In Section \ref{sec:experiments}, we compare some of these different choices on a range of data sets and learning tasks. We also demonstrate the stability of computational persistence experimentally.\\

We conclude this section with a simple complexity analysis. Given a $d$-dimensional grid $P$ with $p$ pixels of resolution in each direction, the total number of voxels $\vert P\vert$ is equal to $p^d$. Each voxel participates in a uniformly bounded number of higher-dimensional simplices, so the total number of simplices in the resultant complex is $O(p^d)$. Note that this is different from simplicial complexes built on top of point clouds, in which each point is allowed to create simplices with every other point; these non-local connections cause a further combinatorial explosion in the number of simplices. Thus, for functional persistence on grids, the total complexity of computing the persistence of a function on such a complex is $O(p^{d\omega})$, where $\omega$ is the matrix multiplication constant. When moving to convolutional persistence, the resolution shrinks when the stride is bigger than one. For a stride vector $k=(2,\cdots, 2)$, the total number of voxels in the resulting grid $R$ is $O(\left(\frac{p}{2}\right)^d)$, and hence the persistence calculations are $O(\left(\frac{p}{2}\right)^{d\omega})$, so that the complexity bound has decreased by a factor of $2^{d\omega}$. More generally, setting $\kappa = \Pi_{i} k_{i}$, computing convolutional persistence for $M$ filters has complexity $O(M\left(\frac{\vert P\vert}{
\kappa}\right)^{\omega})$.  Additionally, increasing the size of $B$ decreases the size of $R$, as fewer translates of $B$ sit inside $P$; to be precise, increasing the length of $B$ in a given coordinate shrinks the resulting region $R$ by an identical amount in the same coordinate. Altogether, this means that, even computing convolutional persistence for a large class of filters, the speed-ups afforded by downsampling are significant, especially when using large stride vectors. This was verified experimentally in \citep{solomon2021fast}, in the context of optimization of topological functionals.

\subsection{Simplicial Convolutional Persistence}

 Let $K$ be a connected simplicial complex, and $f:K^{0}\to \mathbb{R}^{d}$ a function on the vertices of $K$, which can be extended via the lower-* filtration to a function $\hat{f}$ on all of $K$. If $\vert K^{0}\vert = n$, we can encode the function $f$ by a $n \times d$ matrix $X$ by imposing an order on the vertices of $K$ and setting the $i$th row of $X$ to be the value of $f$ on the $i$th vertex. We can also encode the $1$-skeleton $K^{1}$ of the complex as a graph with adjacency matrix $A$. Following \citep{kipf2016semi}, for a weight matrix $W$ of shape $d \times k$, we can transform the data $X$ via the map $X \to AXW$, which amounts to transforming the data $X$ by first multiplying by $W$, and then defining the value at a vertex to be the average of the values at its neighbors (adding a self-loop lets us include the vertex itself in this set). When $W = w$ has shape $d \times 1$, $AXW$ encodes a map from $K^{0}$ to $\mathbb{R}$. We now define simplicial convolutional persistence by generalizing this construction and incorporating persistent homology.

\begin{definition}
	Let $K$ be a simplicial complex and $f: K^{0} \to \mathbb{R}^d$ a function on the vertices of $K$, encoded by an $n \times d$ matrix $X$. Let $A$ be an arbitrary, but fixed, $n \times n$ matrix, not necessarily the adjacency matrix of $K^{1}$, and $W = w$ a $d \times 1$ weight matrix. We define the \emph{simplicial convolutional persistence} of $(K,A,f)$ on $w$ to be $\operatorname{PH}(K,\widehat{f \ast w})$, where $f \ast w$ is the function on $K^{0}$ associated to $AXw$, and $\widehat{f \ast w}$ is its lower-* extension to all of $K$.
\end{definition}

We take $A$ to be an arbitrary matrix to allow for graph structures with self-loops, weights, and asymmetric distances. It might seem from the definition above that there is no need to separate the data matrix $X$ from the generalized adjacency matrix $A$, and so we might consider only the product $AX$, as this is what appears in all calculations. However, distinguishing $X$ from $A$ has conceptual value, as $X$ encodes the data defined on the simplicial complex, and $A$ defines the complex itself. Thus, for example, noise in our measurements ends up in matrix $X$ rather than the matrix $A$, and this is important for understanding the stability of convolutional persistence. 

\begin{remark}
	As simplicial convolutional persistence does not shrink the domain of the function, it does not exhibit the computational speedups of image convolutional persistence.
\end{remark}

\subsection{Prior Work}
\label{sec:priorwork}
Convolutional persistence is a type of \emph{topological transform}, in which a parametrized family of topological invariants is associated to a fixed input. The first topological invariants studied were the \emph{persistent homology transform} (PHT) and \emph{Euler characteristic transform} (ECT) \citep{turner2014persistent}. In the same paper in which the PHT and ECT were defined, inverse theorems were proven; these theorems were generalized and extended by later work \citep{curry2018many,ghrist2018persistent}. We will later show that convolutional persistence is, generically, a special case of the PHT, allowing us to take advantage of the inverse theory developed for that invariant. Invertible topological transforms have also been developed for weighted simplicial complexes \citep{jiang2020weighted}, metric graphs \citep{oudot2017barcode}, metric spaces \citep{solomon_et_al:LIPIcs.SoCG.2022.61}, and metric measure spaces \citep{maria2019intrinsic}. Consult \citep{oudot2020inverse} for a more thorough survey of inverse problem in applied topology.\\

Prior research has also studied the interaction of persistent homology and image convolutions. In \citep{solomon2021fast}, the authors use convolutions to stabilize and speed up topological optimization; the approach taken there can be viewed as a special case of convolutional persistence, in which the filter set $\mathcal{G}$ consists of many random approximations of the box-smoothing filter. Unlike in this work, the goal of \citep{solomon2021fast} is to produce many downsampled images with persistence similar to the original, permitting the computation of robust topological gradients. Another paper that considers both persistence and convolutions is \citep{kim2020pllay}, where the authors develop a topological layer for deep learning models; one such model they consider is a convolutional neural network, and the authors design the network to compute topological features both before and after computing convolutions. However, \citep{kim2020pllay} do not study the interaction of convolutions and persistence in any generality. \\

\section{Theoretical Results}
\label{sec:theoreticalresults}
In what follows, we prove two versions of each result, one for each flavor of convolutional persistence.

\subsection{Stability} 

\begin{remark}
	We remind readers that Theorem \ref{thm:STwasstab}  \citep{skraba2020wasserstein} is an unpublished result, and that this Theorem is needed to establish the $p \neq \infty$ cases of the following two propositions.
\end{remark}

We begin with image convolutional persistence. Fix an input domain $P$, filter region $B$ and stride vector $k$, and a skeleton dimension $m$. Focusing on single-channel images, this determines the domain $R$ for any convolutions of $f: P \to \mathbb{R}$ and $g: B \to \mathbb{R}$.

\begin{proposition}
\label{prop:stability}
Let $f_1,f_2: P \to \mathbb{R}$ be two functions on $P$, and $g:B \to \mathbb{R}$ a filter function. Then, computing persistence over the complex $K_{R}^m$, for any $p\geq 1$ we have:
\begin{align*}
	W_{p}(\operatorname{Diag}(\widehat{f_1 \ast g}),\operatorname{Diag}(\widehat{f_2 \ast g}))  & \leq\\ (3^{m}-2^{m})\|g\|_{1} \|f_1 - f_2\|_{p} & \leq\\ \vert B \vert (3^{m}-2^{m})\|g\|_{\infty} \|f_1 - f_2\|_{p}.
\end{align*}

When $p = \infty$, the bound can be made tighter:
\begin{align*}
	W_{\infty}(\operatorname{Diag}(\widehat{f_1 \ast g}),\operatorname{Diag}(\widehat{f_2 \ast g}))  & \leq\\ \|g\|_{1} \|f_1 - f_2\|_{\infty} & \leq\\ \vert B \vert \|g\|_{\infty} \|f_1 - f_2\|_{\infty}.
\end{align*}

\end{proposition}
\begin{proof}Young's convolutional inequality\footnote{See \citep{hewitt2012abstract}, Theorem 20.18 on page 296, for a proof of Young's inequality for locally compact groups, which in particular includes $\mathbb{Z}^d$.} states that for $ 1 \leq p,q,r \leq \infty$ with $\frac{1}{p} + \frac{1}{q} = \frac{1}{r} + 1$,
\[\|f \ast g\|_{r} \leq  \|g\|_{q} \|f\|_{p}.\]
Setting $r=p$ forces $q=1$, and setting $f = f_1 - f_2$, we obtain the bound $\|f_1 \ast g - f_2 \ast g\|_{p} \leq \|g\|_{1}\|f_1 - f_2\|_{p}$. Since $g$ is supported on a set of $\vert B \vert$ elements, we have $\|g\|_{1} \leq \vert B\vert\|g\|_{\infty}$, so that $\|f_1 \ast g - f_2 \ast g\|_{p} \leq \vert B \vert \|g\|_{\infty}\|f_1 - f_2\|_{p}$. The bound then follows from Theorem \ref{thm:STwasstab} in general, and from Corollary \ref{cor:infstabcor} when $p = \infty$.
\end{proof}

 We now consider simplicial convolutional persistence.

\begin{definition}
	Let $X$ be an $n \times k$ matrix. For $1 \leq p \leq \infty$, we define the $p$-operator norm of $X$ as follows:
	\[\|A\|_{p}^{*} = \sup_{x \neq 0} \frac{\|A x\|_{p}}{\|x\|_{p}}. \]
It is a standard result that operator norms are submultiplicative, meaning that $\|X_{1}X_{2}\|_{p}^{*} \leq \|X_{1}\|_{p}^{*}\|X_2 \|_{p}^{*}$, where $X_1$ and $X_2$ are any two matrices that can be multiplied.
\end{definition}

\begin{proposition}
\label{prop:simpstability}
Let $K$ be a simplicial complex and $f_{1},f_{2}: K^{0} \to \mathbb{R}^{d}$ two functions on its vertex set, encoded by matrices $X_1$ and $X_2$. Write $n = \vert K^{0} \vert$, and take $A$ to be the $n \times n$ matrix used in convolutional persistence. Take a weight vector $w \in \mathbb{R}^d$. Then, computing persistence over the complex $K^m$, for any $p\geq 1$ we have:
\begin{align*}
	W_{p}(\operatorname{Diag}(\widehat{f_1 \ast w}),\operatorname{Diag}(\widehat{f_2 \ast w}))  & \leq\\ (3^{m}-2^{m})\|A\|_{p}^{*} \|X_{1} - X_{2}\|_{p}^{*}\|w\|_{p}.
\end{align*} 

When $p = \infty$, the bound can be made tighter:
\begin{align*}
	W_{\infty}(\operatorname{Diag}(\widehat{f_1 \ast w}),\operatorname{Diag}(\widehat{f_2 \ast w}))  & \leq\\ \|A\|_{\infty}^{*} \|X_{1} - X_{2}\|_{\infty}^{*}\|w\|_{\infty}.
\end{align*} 
\end{proposition}
\begin{proof}
The difference $f_{1} \ast w - f_{2} \ast w$ can be written as $A(X_1 - X_2)w$. Submultiplicativity of the $p$-operator norm shows that $\|A(X_1 - X_2)w\|_{p}^{*} \leq \|A\|_{p}^{*}\|(X_1 - X_2)\|_{p}^{*}\|w\|_{p}^{*}$. For column vectors $v$, it is easy to see that $\|v\|_{p}^{*} = \|v\|_{p}$, i.e. the $p$-operator norm is equal to the usual $p$-norm. Thus we can write
\[\|f_1 \ast w - f_2 \ast w\|_{p} = \|A(X_1 - X_2)w\|_{p} \leq \|A\|_{p}^{*}\|(X_1 - X_2)\|_{p}^{*}\|w\|_{p}.\]
As before, the bound then follows from Theorem \ref{thm:STwasstab} in general, and from Corollary \ref{cor:infstabcor} when $p = \infty$.
\end{proof}

As all matrix norms are equivalent, the quantity $\|X_1 - X_2\|_{p}^{*}$ is bounded above and below by a constant multiple of $\|X_1 - X_2\|_{p}$, the matrix $p$-Frobenius norm, providing for stability in this norm too.

\subsection{Inverse Theory}

 We begin this section by recalling the definition of the persistent homology transform and some of its injectivity properties. We then prove that the persistent homology transform can be extended to functions defined on abstract simplicial complexes, and that the resulting transform also has strong inverse properties. Finally, we use this new theorem to prove inverse results for our two convolutional persistence transforms.

\begin{definition}[\citep{turner2014persistent}]
	Let $M \subset \mathbb{R}^d$ be a finite simplicial complex. For every vector $v \in \mathbb{S}^{d-1}$, we can define the function $f_{v}(x) = \langle x,v \rangle$. The persistent homology transform $PHT(M)$ is then defined as the map from $\mathbb{S}^{d-1}$ to the space of persistence diagrams that send a vector $v$ to the persistence diagram of $(S,f_{v})$. One can define a similar invariant using Euler curves instead of persistence diagrams, called the Euler characteristic transform (ECT).
\end{definition}

The PHT was shown to be injective for $d=2,3$ in \citep{turner2014persistent}. Later work \citep{ghrist2018persistent,curry2018many} proved injectivity in all dimensions, for both the PHT and ECT, for a very general class of subsets that includes simplicial complexes.\\

\begin{theorem}[\citep{turner2014persistent,ghrist2018persistent,curry2018many}]
	\label{thm:phtinj}
	Let $K,L \subset \mathbb{R}^d$ be two embedded simplicial complexes. If $PHT(K) = PHT(L)$ or $ECT(K) = ECT(L)$ then $K=L$.
\end{theorem}

Now, given a complex (simplicial or cubical) $K$ with a function $f:K^{0} \to \mathbb{R}^d$, we can define an analogous topological transform. 

\begin{definition}
	Let $K$ be a complex (simplicial or cubical) with a function $f:K^{0} \to \mathbb{R}^d$ defined on its vertices, extending to a function $\widehat{f}$ on $K$ via the lower-* filtration. For $v \in \mathbb{S}^{d-1}$, we can define the function $\widehat{f} \cdot v : K \to \mathbb{R}$. We then define the PHT of $(K,f)$ to be the map on $\mathbb{S}^{d-1}$ that sends a vector $v \in \mathbb{S}^{d-1}$ to the persistence diagram of $(K,\widehat{f} \cdot v)$. We define the ECT similarly.
\end{definition}

We now prove an injectivity result for the PHT and ECT of such complex-function pairs $(K,f)$. This result shows that the PHT and ECT are injective up to a simple equivalence relation which we now define.

\begin{definition}
	For a complex $K$ with a function $f:K^{0}\to \mathbb{R}^d$, we write $f^{\triangle}$ to denote the function $f^{\triangle}: \operatorname{geom}(K) \to \mathbb{R}^d$ defined on the geometric realization of $K$ via linear interpolation. This differs from $\widehat{f}$, which is defined via the discrete lower-* rule. Note that we use the notation $\operatorname{geom}(K)$ as opposed to the more traditional $\vert K \vert$ to avoid confusion with the cardinality of $K$ as a set.
\end{definition}

\begin{definition}
	Let $(K,f)$ and $(L,g)$ be two complexes with $\mathbb{R}^d$-valued functions defined on their vertex sets. We will say that $(K,f)$ and $(L,g)$ are \emph{re-discretization equivalent} if there is a homeomorphism $h: \operatorname{geom}(K) \to \operatorname{geom}(L)$ with $f^{\triangle} = g^{\triangle} \circ h$. In other words, $(K,f)$ and $(L,g)$ induce the same geometric object and locally linear map.
\end{definition}

\begin{theorem}
	\label{thm:simpphtinj}
	Suppose that $(K,f)$ and $(L,g)$ are two complexes with $\mathbb{R}^d$-valued functions defined on their vertex sets, and suppose further that $f^{\triangle}$ and $g^{\triangle}$ are injective embeddings of $\operatorname{geom}(K)$ and $\operatorname{geom}(L)$ into $\mathbb{R}^d$. Then $PHT(K,f) = PHT(L,g)$ or $ECT(K,f) = ECT(L,g)$ implies that $f^{\triangle}(\operatorname{geom}(K)) = g^{\triangle}(\operatorname{geom}(L))$, i.e. $(K,f)$ and $(L,g)$ are re-discretization equivalent, with $h = g^{\triangle} \circ (f^{\triangle})^{-1}$.
\end{theorem}
\begin{proof}
For a weight vector $w \in \mathbb{S}^{d-1}$, the functions $\hat{f} \cdot w$ and $f^{\triangle}$ are identical on the vertex set $K^{0}$, but differ in their extensions to $K$ or $\operatorname{geom}(K)$, as the former is discrete and the latter continuous. However, Proposition \ref{prop:twofilt} shows that they give the same persistence diagrams (and hence, also, Euler characteristic curves). It is, moreover, easy to see that the filtration on $\operatorname{geom}(K)$ given by $f^{\triangle} \cdot w$ is identical to the filtration on $f^{\triangle}(\operatorname{geom}(K))$ given by $f_{w}(x) = x \cdot w$, identifying $\operatorname{geom}(K)$ with $f^{\triangle}(\operatorname{geom}(K))$ by dint of the fact that $f^{\triangle}$ is a homeomorphism. Thus, the PHT of $(K,f)$ contains the same data as the PHT of $f^{\triangle}(\operatorname{geom}(K))$ (and the same is true for the ECT), and so we can apply Theorem \ref{thm:phtinj}.
\end{proof}

In order to use this theorem to prove inverse results for convolutional persistence, we need one more technical lemma.

\begin{definition}
	A collection of points $S$ in $\mathbb{R}^d$ are in \emph{general position} if no hyperplane intersects $S$ in more than $d$ points.
\end{definition}

\begin{remark}
	It is easy to see that the property of being in general position is invariant under translation. Moreover, if a set $S$ in general position in $\mathbb{R}^d$ contains the origin as a point $s_0 = \vec{0} \in S$, then any set of $d$ vectors in $S \setminus s_0$ is linearly independent.
\end{remark}

\begin{remark}
	Being in general position is a generic property of point sets, meaning that the subset of $(\mathbb{R}^d)^{S}$ corresponding to subsets in general position is open and dense.
\end{remark}

\begin{lemma}
	\label{lem:genposemb}
	Let $S$ be the vertex set of an $m$-dimensional simplicial complex $K^{m}$. Let $\iota : S \hookrightarrow \mathbb{R}^{M}$ be an embedding with $\iota(S)$ in general position. Then if $M \geq 2m+1$, the embedding extends by interpolation to an injective embedding $\iota^{\triangle}:K^{m} \hookrightarrow \mathbb{R}^{M}$. Similarly, let $S$ is vertex set of an $m$-dimensional cubical complex $L^{m}$ and  $\iota : S \hookrightarrow \mathbb{R}^{M}$ an embedding with $\iota(S)$ in general position. Then if $M \geq 2^{m+1}-1$, the embedding extends by interpolation to an injective embedding $\iota^{\triangle}:L^{m} \hookrightarrow \mathbb{R}^{M}$.
\end{lemma}
\begin{proof}
	We begin with the simplicial complex. A nontrivial intersection of two faces $\iota(\sigma_1)$ and $\iota(\sigma_2)$ for $\sigma_1,\sigma_2 \leq K^{m}$ implies that the union of their vertex sets $\iota(\sigma_1^{0} \cup \sigma_2^{0})$ is not linearly dependent (this includes the case when the two faces are identical, i.e. a self-intersection), and in fact this dependence relation has at least two nonzero terms, as a dependence relation with only one nonzero term corresponds to a vertex sitting at the origin, which does not reflect any intersection of faces. Translating one of the vertices in $\iota(\sigma_1^{0} \cup \sigma_2^{0})$ to the origin and dropping it from the dependence relation, we are still left with a nontrivial dependence relationship among the remaining vertices. Now, a face of $K^{m}$ has at most $m+1$ vertices, so the union of two faces involves at most $2m+2$ vertices. If $M \geq 2m+1$, translating one of these vertices to the origin means that the remaining vertices are linearly independent, which contradicts the possibility of such a dependence relation.\\
	
	The proof is similar with the cubical complex. A face of $L^m$ contains at most $2^{m}$ vertices, so the union of two faces contains at most $2^{m+1}$ vertices. If $M \geq 2^{m+1}-1$, translating one of these $2^{m+1}$ points to the origin makes the remaining $2^{m+1}-1$ linearly independent, an hence no intersection can occur.
\end{proof}

We now prove our inverse results for convolutional persistence.

\subsection*{Simplicial Convolutional Persistence}

Recall that for a function $f: K^{0} \to \mathbb{R}^d$ defined on the vertex set of simplicial complex and encoded by an $\vert K^{0} \vert \times d$ matrix $X$, an arbitrary fixed $\vert K^{0} \vert \times \vert K^{0} \vert$ matrix $A$, and a $d \times 1$ weight vector $w$, the simplicial convolutional persistence of the triple $(K,A,f)$ on $w$ is defined by first writing $f \ast w$ for the function on $K^{0}$ defined by $AXw$ and then computing the persistent homology of the lower-star extension $\widehat{f \ast w}$ on $K$. We now define a topological transform based on this convolution.

\begin{definition}
	Fix $f:K^{0} \to \mathbb{R}^d$ and $A$ as in the definition of simplicial convolutional persistence. Define the convolutional persistence transform to be the mapping $CPT(K,f,A):\mathbb{S}^{d-1} \to \mathbf{Diagrams}$ that sends a weight vector $w \in \mathbb{S}^{d-1}$ to the persistence diagram $\operatorname{PH}(K,\widehat{f \ast w})$. One can similarly define an \emph{Euler characteristic curve} version of this invariant, the \emph{Convolutional Euler Characteristic Transform} (CECT), by computing Euler characteristic curves instead of persistence diagrams.
\end{definition}

\begin{definition}
	We will say that a property is \emph{generically} true if it is true when restricted to some open, dense set of parameters. For example, with simplicial convolutional persistence, the parameters are the function $f$, encoded by the matrix $X$, and the generalized adjacency matrix $A$.
\end{definition}

\begin{proposition}
	Assume that $d \geq \operatorname{dim}(K)+1$. Then $f^{\triangle}$ is generically an injective embedding of $\operatorname{geom}(K)$, and so the CPT and CECT of $(K,f,A)$ determine $(K,f)$ \emph{up to re-discetization}.
\end{proposition}
\begin{proof}
It is a generic property of the matrix $AX$ that its rows are in general position in $\mathbb{R}^d$. Lemma \ref{lem:genposemb} then implies that $f^{\triangle}$ is injective on the geometric realization. We can then apply Theorem \ref{thm:simpphtinj}.
\end{proof}

\subsection*{Image Convolutional Persistence}
We begin by defining our topological transform. 

\begin{definition}
	Fix $f: P \to \mathbb{R}$. The \emph{Convolutional Persistence Transform}  is the mapping $CPT(f): \mathbb{S}^{\vert B\vert-1} \to \mathbf{Diagrams}$ that sends an $L^2$-normalized function $g: B \to \mathbb{R}$ to the persistence diagram of $f \ast g$ on $K_{R}^{m}$. One can similarly define an \emph{Euler characteristic curve} version of this invariant, the \emph{Convolutional Euler Characteristic Transform} (CECT), by computing Euler characteristic curves instead of persistence diagrams. Using multi-channel images $f: P \to \mathbb{R}^n$, we can define an analogous CPT and CECT. In this setting, $g$ ranges over all normalized functions $B \to \mathbb{R}^n$, which can be identified with $\mathbb{S}^{\vert B\vert n-1}$.
\end{definition}

\begin{remark}
	Taking a trivial convolution where $B$ consists of a single vertex $b$ and $g(b)=1$, we can recover the original persistence diagram of $f$. Thus, the CPT is a strictly more general construction than ordinary persistence. 
\end{remark}

Next, we show how to associate $(P,f)$ with a Euclidean embedding of $K_{R}$.

\begin{definition}
	For a fixed function $f:P \to \mathbb{R}^n$, we obtain a mapping $\iota_{f}$ of the rectangle $R$ into $\mathbb{R}^{\vert B\vert n}$ by sending every point $r \in R$ to the vector $\{f(b + k \odot r) \mid b \in B\}$. In other words, $r$ corresponds to a translate of $B$, and $\iota_{f}(r)$ records the values of $f$ restricted to this translate. This is technically a set, rather than a vector, but it becomes a vector after fixing an order on the elements of $B$. Such a mapping can be extended via the lower-* rule to a map $\widehat{\iota_{f}}:K_{R} \to \mathbb{R}^{\vert B \vert n}$ on the entire cubical complex $K_{R}$ built on top of $R$. See Figure \ref{fig:embedding}.
\end{definition}

\begin{figure}
	\centering
	\includegraphics{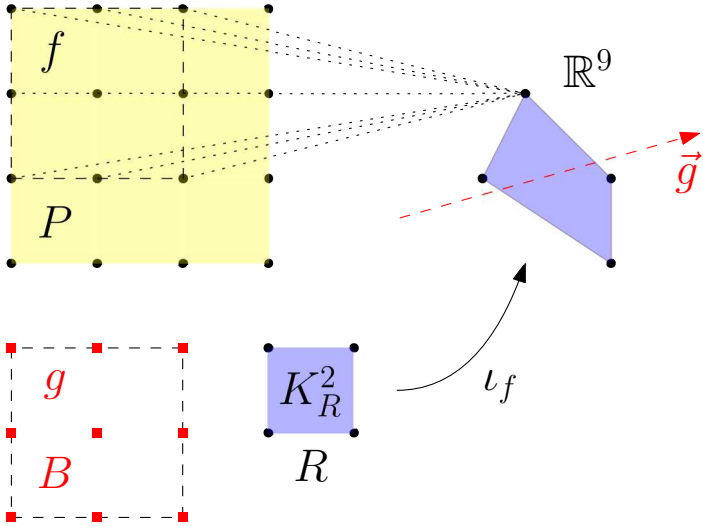}
	\caption{The function $f$ is defined on the $4 \times 4$ grid on the top left. The box $B$ is $3 \times 3$, and using a $(1,1)$-stride there are four ways of laying $B$ over the domain of $f$, so that $R$ is a $2 \times 2$ grid. We can map the vertices of $R$ into $\mathbb{R}^9$ by associating each vertex of $R$ with its corresponding translate of $B$, and then taking as coordinates the values of $f$ in that translate. This extends via interpolation to a map $\iota_{f}$ from the complex $K_{R}^2$ into $\mathbb{R}^9$, which here is shown to be an embedding.}
	\label{fig:embedding}
\end{figure}

\begin{figure}
	\centering
	\includegraphics{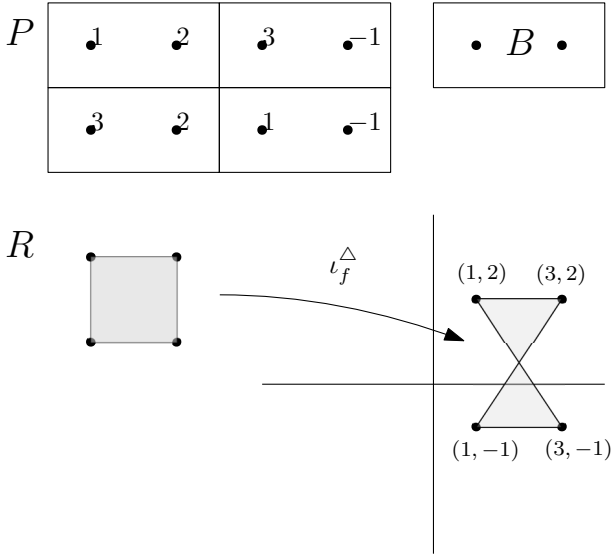}
	\caption{The grid $P$ has shape $2 \times 4$, whereas $B$ has shape $1 \times 2$. Using a convolutional of stride $(2,1)$, there convolutional grid $R$ has shape $2 \times 2$. The values of the function $f$ are indicated in the grid $P$. This induces a mapping $\iota_{f}$ from the vertices of $R$ into $\mathbb{R}^2$, where the top-left vertex of $R$ gets sent to $(1,2)$, the top-right to $(3,-1)$, the bottom-left to $(3,2)$ and the bottom-right to $(1,-1)$. Extending this via interpolation to the complex $K_{R}^{2}$ results in some self-intersections, so that $\iota_{f}^{\triangle}$ is not a homeomorphism on to its image.}
	\label{fig:PCT}
\end{figure}

It is easily seen that, fixing the rectangles $P$ and $B$ and the stride vector $k$, that $(K_{R},\widehat{\iota_{f}})$ uniquely determines $f$, and moreover that $CPT(f) = PHT(K_{R},\widehat{\iota_{f}})$. To obtain an inverse result in this context, we must show that the continuous interpolation $\iota_{f}^{\triangle}$ is generically injective on $\operatorname{geom}(K_{R}^m)$. See Figure \ref{fig:PCT}.

\begin{lemma}
	Let $\kappa = \Pi_{i}k_i$, where $k_i$ is the $i$th entry of the stride vector $k$, and suppose that $n\kappa  \geq 2^{m+1}-1$. Then a generic function $f$ has the property that  $\iota_{f}^{\triangle}$ is injective on the $m$-skeleton $K_{R}^m$, so that $\iota_{f}^{\triangle}(\operatorname{geom}(K_{R}^m))$ has the structure of a cubical complex isomorphic to $K_{R}^m$.
	\label{lem:genemb}
\end{lemma}
\begin{proof} Let $B^{*} \subseteq B$ consist of those elements in the top-left $k_1 \times k_2 \times \cdots k_d$ corner of $B$. The various translates $B^{*} + (k \odot r)$ are all disjoint subsets of $P$. To show that $\iota_{f}^{\triangle}$ is injective, it suffices to show that it is injective when composed with the coordinate projection $\pi: \mathbb{R}^{\vert B \vert n} \to \mathbb{R}^{\vert B^{*} \vert n}$. However, since the translates of $B^{*}$ are disjoint, a choice of $\pi \circ \iota_{f}$ is equivalent to choosing an arbitrary vector in $\mathbb{R}^{\vert B^* \vert n}$ for each pixel in $R$ (a collection which will generically be in \emph{general position}), and the map $\pi \circ \iota_{f}^{\triangle} = (\pi \circ \iota_{f})^{\triangle}$ is obtained by linearly interpolating this mapping on the interior of the higher-order simplices of $K_{R}^{m}$. By Lemma \ref{lem:genposemb}, if $\vert B^{*} \vert n \geq 2^{m+1}-1$, we know that $\pi \circ \iota_{f}^{\triangle}$ is injective, and hence $\iota_{f}^{\triangle}$ is injective.
\end{proof}

\begin{theorem}
	Fix $P,B,k$, and assume that $\kappa = \Pi_{i}k_i \geq 2^{m+1}-1$. Then, generically, $CPT(f) = CPT(g)$ implies $f=g$, and the same is true for the CECT.
	\label{thm:thmcptim}
\end{theorem}
\begin{proof}
	Lemma \ref{lem:genemb} and Theorem \ref{thm:simpphtinj} imply that $CPT(f)$ or $CECT(f)$ determines $\iota_{f}^{\triangle}(\operatorname{geom}(K_{R}^m))$, which uniquely determines $f$.
\end{proof}

\begin{remark}
What we have just shown is that having a large stride vector, in addition to providing computational speedups by lowering the resolution of the resultant grid, also provides generic injectivity for persistence of higher-dimensional data complexes.
\end{remark}

The following corollary shows that to apply Theorem \ref{thm:thmcptim} one can ignore those filter functions orthogonal to image patches in our dataset. As \citep{carlsson2008local} observed, the space of local patches in natural images has very high codimension in the space of all possible patches, and hence in practice the CPT may not require a very high-dimensional collection of filter functions to be effective at distinguishing images.

\begin{definition}
For a given function $f:P \to \mathbb{R}^{d}$, we say that a filter function $g:B \to \mathbb{R}^{d}$ is \emph{convolutionally orthogonal} to $f$ if $f \ast g = 0$, i.e. $(f \ast g)(r) = 0$ for all $r \in R$. We write $f^{\ddagger}$ to indicate all functions convolutionally orthogonal to $f$.
\end{definition}

\begin{corollary}
\label{cor:orthoginv}
Let $f_1, f_2 : P \to \mathbb{R}^d$ be two functions, and write $\mathcal{G} = (f_{1}^{\ddagger} \cap f_{2}^{\ddagger})^{\perp}$, the orthogonal complement of the subspace of all functions convolutionally orthogonal to both $f_1$ and $f_2$. If $\operatorname{Diag}(\widehat{f_1 \ast g}) = \operatorname{Diag}(\widehat{f_2 \ast g})$ for all $g \in \mathcal{G}$, then $CPT(f_1) = CPT(f_2)$, and the same is true replacing persistence diagrams with Euler characteristic curves.
\end{corollary}
\begin{proof}
If $f_1$ and $f_2$ have the same persistence diagram when convolving with a function $g \in \mathcal{G}$, they also have the same diagram after convolving with $g+h$ for $h \in \mathcal{G}^{\perp}$, due to the  linearity of convolution.

\end{proof}

\section{Experiments}
\label{sec:experiments}

In this section, we consider how convolutional persistence compares with ordinary persistence in machine learning applications. We focus on image classification, which involves image convolutional persistence, leaving experiments with simplicial convolutional persistence for future work. Our aim is not to argue that persistence-based methods are superior to other methods for the tasks at hand.
Nor we do try to demonstrate the computational advantages of downsampling, as this has already been shown in \citep{solomon2021fast}, so we use convolutions where the original image and the convolved image have the same dimensions, i.e. a unit stride vector $(1,1)$. In all our computations, we build a simplicial complex wherein the pixels are the top-dimensional simplices, rather than the vertices; this is the second construction in Section \ref{sec:imagecomplex}, and is in line with how the Giotto toolkit computes cubical persistence \citep{tauzin2021giotto}. Code for running experiments with convolutional persistence can be found at \href{https://github.com/yesolomon/convpers}{https://github.com/yesolomon/convpers}.\\

\subsection*{Classification Tasks}
In what follows, we consider five  datasets:
\begin{enumerate}
\item The UCI digits dataset \citep{alpaydin1998optical}. The images are $8 \times 8$. This dataset has 5620 images.
\item The MNIST dataset \citep{deng2012mnist}. The images are $28 \times 28$, and $5000$ random images were chosen for the dataset.
\item An MNIST-like dataset of Chinese digits \citep{Nazarpour2017}. There are 15 classes, for a baseline accuracy of $1/15 \approx 0.066$. See Figure \ref{fig:chinese_examples} for examples. To speed up computations, a $2 \times 2$ max pooling is applied before computing persistence. The original images are $64 \times 64$, and are $32 \times 32$ after max pooling. $5000$ random images were chosen the dataset.
\item An MNIST-like dataset of characters in the Devanagari script \citep{acharya2015deep}. There are 46 classes, for a baseline accuracy of $1/46 \approx 0.0217$. See Figure \ref{fig:devanagari_examples} for examples. The images are $32 \times 32$, and $5000$ random images were chosen for the dataset.
\item A dataset of solutions to the \emph{Kuramoto-Sivashinsky} PDE. This is a PDE of relevance in many systems exhibiting pattern formation \citep{cuerno1995dynamic,motta2012highly,villain1991continuum,wolf1991kinetic,golovin1998effect}. This PDE was studied in the context of topological machine learning by Adams et al. \citep{adams2017persistence}, in which they consider the anisotropic form of the PDE:
\[u_{t} = -\nabla^2 u - \nabla^2 \nabla^2 u + r (u_x)^2 + (u_y)^2.\]
Starting from random initial conditions, we solve these PDEs out to time $t=15$ for a range of $r$ values, $r \in \{1,1.25,1.5,1.75,2\}$. The classification task is then parameter estimation: giving a 2D image of the solution to the PDE, guess the $r$ parameter that produced it. See Figure \ref{fig:KS_examples} for examples. The images are $50 \times 50$, and $500$ images were generated for each $r$ value, giving a dataset of $2500$ points.
\end{enumerate}

\begin{figure}
\centering
\includegraphics[scale=0.35]{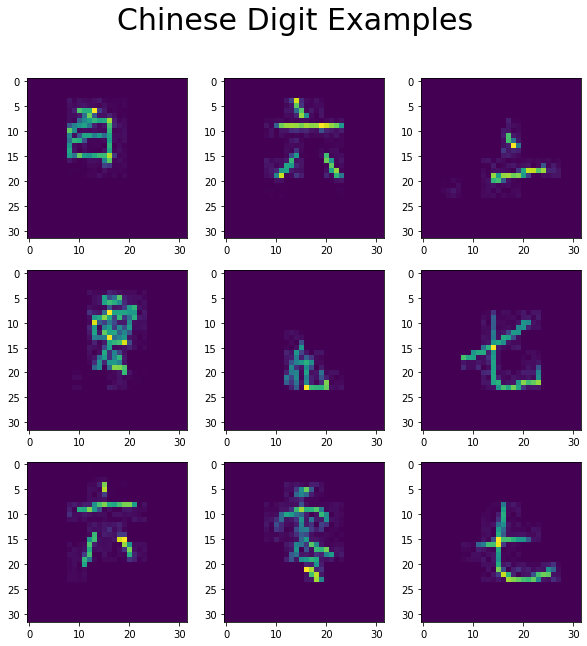}
\caption{Examples of chinese digit figures.}.
\label{fig:chinese_examples}
\end{figure}

\begin{figure}
\centering
\includegraphics[scale=0.35]{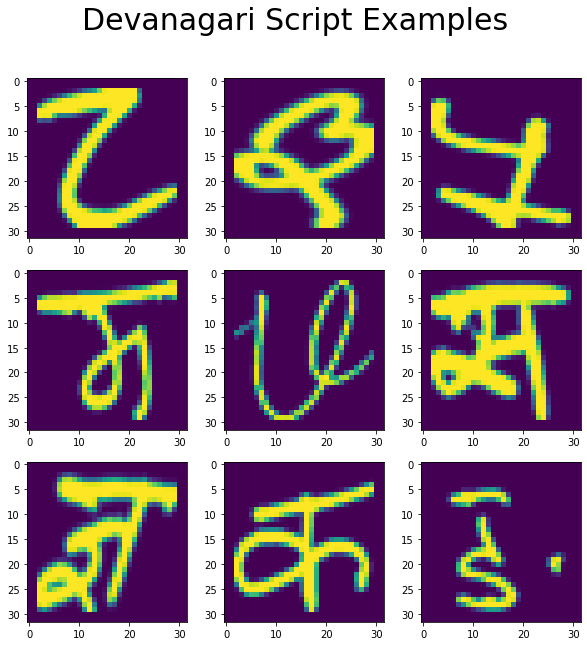}
\caption{Examples of Devanagari script figures.}.
\label{fig:devanagari_examples}
\end{figure}

\begin{figure}
\centering
\includegraphics[scale=0.23]{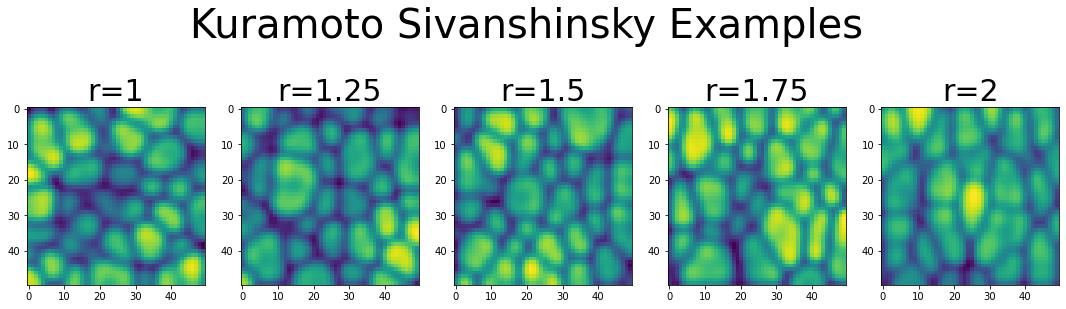}
\caption{Examples of Kuramoto-Sivashinsky figures. Note how the cells become vertically stretched for larger values of $r$.}.
\label{fig:KS_examples}
\end{figure}

In all of our datasets, the images are small enough that the convolutions are defined to preserve the resolution, rather than downsample. The goal here is to ignore the computational advantages of convolutional persistence and focus only on the gains in discriminative power. For each dataset, we consider five versions of the CPT:
\begin{itemize}
\item Using a trivial $1 \times 1$ filter $[1]$, which is equivalent to computing the persistence of the original image.
\item Adding, in addition to the trivial filter, three more ``standard" filters: sharpening, blur, and Gaussian kernels, all $3 \times 3$.
\item Applying PCA to the space of $3 \times 3$ image patches to obtain the top-three principal component patches, and then generating five kernels as random normalized linear combinations of these patches. This approach is motivated by Corollary \ref{cor:orthoginv}, which says that, for the purposes of injectivity, one can ignore filters orthogonal to those found in the patches of the dataset. We call the resulting kernels \emph{eigenfilters}.
\item Taking 5 random, normalized $3 \times 3$ filters.
\item Taking 25 random, normalized $3 \times 3$ filters.
\end{itemize}  

For each version of the CPT, we consider four vectorizations:
\begin{itemize}
\item Vectorize using persistence images (default setting, $10 \times 10$ resolution), concatenating across both homological dimensions and across filters. Thus, if there are $B$ filters, the resulting vector has length $200B$. See Figure \ref{fig:digit-persimg} for a visualization of the persistence images of an image in the Chinese digits dataset.
\item Vectorize using persistence images, as above, but average along filters rather than concatenating, giving a vector of length $200$.
\item Vectorize using total persistence, concatenating across both homological dimensions and across filters. Thus, if there are $B$ filters, the resulting vector has length $2B$.
\item Vectorize using total persistence, as above, but average along filters rather than concatenating, giving a vector of length $2$.
\end{itemize}

\begin{figure}
	\centering
	\includegraphics[scale=0.4]{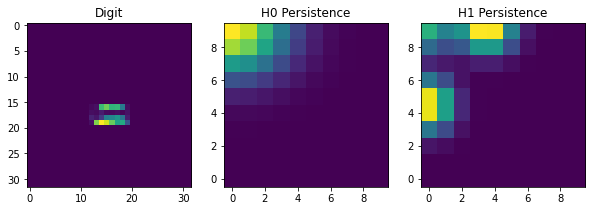}
	\caption{Zero- and one-dimensional persistence diagram for a chinese digit corresponding to ``2". The sublevelset filtrations begins by adding dark regions and then includes regions with increasingly larger values, indicated with brighter colors. Even with the 10x10 resolution, we see clearly that the dark region has one connected component and two persistent cycles.}
	\label{fig:digit-persimg}
\end{figure}

Finally, for each vectorization, we test the topological features on the classification task using the following three models:
\begin{itemize}
\item A $3$-nearest-neighbors classifier.
\item Gradient boosting trees, $10$ estimators.
\item A neural network model with two hidden layers with 100 units each, leakyRELU activations on the hidden layers and softmax on the final layer, and adam optimizer. Trains for 50 epochs with batch size 4. 
\end{itemize}

Though this series of experiments does not fully explore all the choices, applications, or features of the CPT, it is sufficiently diverse to allow for some interesting insights and the observation of recurrent patterns. Moreover, the experimental pipeline chosen is very simple and does not rely on any careful engineering of features or models.\\

In Table \ref{table:computing-time}, we indicate the time it takes to compute the persistence diagrams for all the images in the data set and transform the results into concatenated persistence image vectors. Other steps in the computational pipeline, like performing convolutions, are extremely fast, so the computing time for convolutional persistence using 10 filters is very close to 10 times the values found in the table. Distributing the persistence calculations across different processors, as well as convolving with larger stride vectors (to shrink image resolution) both provide major speedups (as discussed earlier).

\begin{table}[htb!]
	\begin{tabular}{|l|l|l|l|l|l|}
		\hline
		& UCI Digits & MNIST & Chinese Digits & Devanagari Script & KS PDEs \\ \hline
		processing time & 2.4s       & 28.6s & 22.8s          & 21.1s             & 21.6s   \\ \hline
	\end{tabular}
\label{table:computing-time}
\caption{Computing Time for Persistence Calculations}
\end{table}

As many components of the experimental set-up are stochastic, such as the random filters used and the initializations of the neural nets, we perform $10$ simulations for each combination of experimental hyperparameters, setting aside 20\% of the data as a testing set to compute model accuracy. For each choice of filters and vectorization, the model with the highest average is chosen, and the standard deviation of its accuracy across the $10$ simulations is shown in the bar plot. \emph{NN} indicates \emph{nearest neighbors}, \emph{Tree} is \emph{gradient boosting trees}, and \emph{DL} is \emph{deep learning}, i.e. neural networks.

\subsection*{Results}

\begin{itemize}
\item Digits: Convolutional persistence provides great improvements in predictive accuracy. The best results come from concatenating features rather than averaging them. Total persistence is very powerful when many random filters are used. When using only five filters, eigenfilters outperforms random filters. For this dataset, the best average accuracies were achieved using deep learning. See Figure \ref{fig:digits_results}.
\item MNIST: Convolutional persistence against boosts predictive accuracy by a large margin. Similar trends to the digits dataset, although the best models are not always neural networks. See Figure \ref{fig:mnist_results}.
\item Chinese digits: For this dataset, five eigenfilters are significantly better than five random filters. Moreover, the best vectorization is given by concatenating total persistence scores. See Figure \ref{fig:chinese_results}.
\item Devanagari: For this dataset, taking five random filters performs about as well as taking five eigenfilters. The best average performance is given by concatenating persistence images, but similar accuracy with lower standard deviation is given by a nearest neighbor classifier on the concatenated total persistence vectors. See Figure \ref{fig:dev_results}.
\item Kuramoto-Sivashinsky: The advantages of convolutional persistence are evident but less dramatic. As in the prior two experiments, the concatenated total persistence vectors for 25 random filters is significantly more informative than any vectorizations of ordinary persistence. Better performance (over 95\%!) is possible using neural networks on the concatenated persistence image vector, but this is less stable to some of the experimental hyperparameters. See Figure \ref{fig:KS_results}.
\end{itemize}

\subsection*{Observations \& Discussion}

We conclude this experimental section with a few comments.
\begin{itemize}
\item Random filters consistently work well at providing informative features, although there is little theoretical justification for why this should be the case. It is possible that this is related to Johnson-Lindenstrauss theory \citep{lindenstrauss1984extensions}, which shows that random projections are good at preserving the geometry of high-dimensional point clouds.
\item It is also surprising that total persistence is such an effective vectorization across multiple experiments, despite being so reductive.
\item Given that eigenfilters tend to outperform random filters, it would be interesting (and easy) to compare 25 random linear combinations of eigenfilters with 25 random filters.
\item In our eigenfilters experiments, we opt for taking multiple random combinations of the top eigenfilters, rather than using the eigenfilters themselves, and taking more of them. This was based on some theoretical heuristics and small-scale experiments, but there is otherwise no reason to suggest why one approach is generally superior to the other.
\item We did not apply any feature engineering before applying our machine learning models. For example, the concatenated persistence images vectors were fairly large, 5000-dimensional in the case of $25$ random filters, and performance might be improved by pre-applying a dimensionality reduction method like PCA.
\item Similarly, we did not do any engineering of our machine learning models, fixing a simple neural network architecture across all experiments. For the Kuramoto-Sivashinsky PDE dataset, the high variance of the neural network models suggests that significantly improved performance might be achievable using the right architecture and training regime.
\item Another hyper-parameter we did not test extensively is the size of the filters. Some preliminary experiments suggest that $4 \times 4$ filters outperform the $3 \times 3$ filters used here, but it remains to be seen if this phenomenon is robust.
\end{itemize}

\begin{figure}
\centering
\includegraphics[scale=0.35]{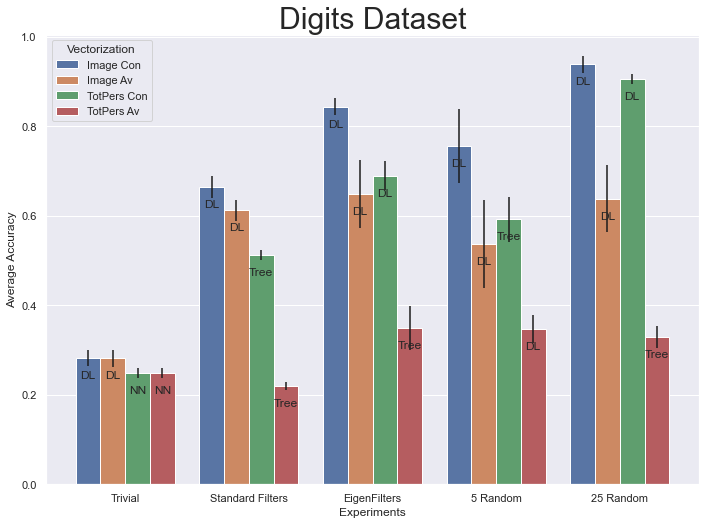}
\caption{Results for the digits dataset.}
\label{fig:digits_results}
\end{figure}

\begin{figure}
\centering
\includegraphics[scale=0.35]{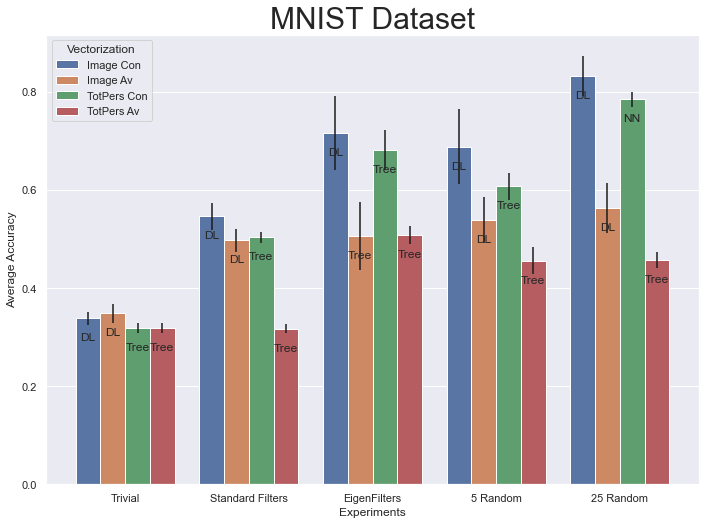}
\caption{Results for the MNIST dataset.}
\label{fig:mnist_results}
\end{figure}

\begin{figure}
\centering
\includegraphics[scale=0.35]{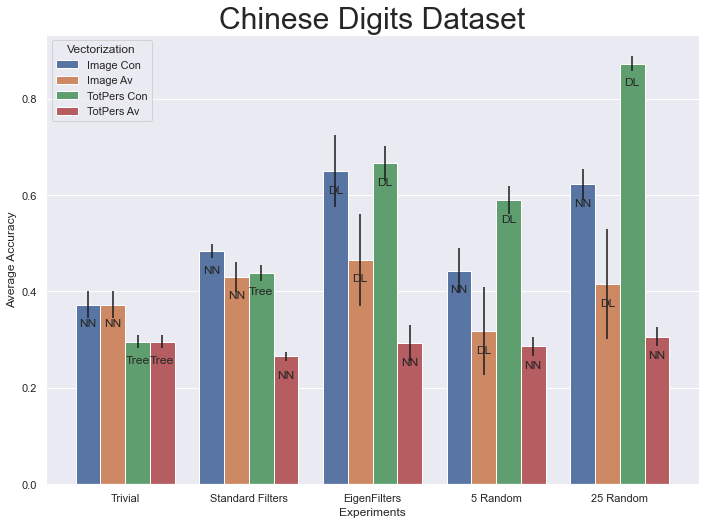}
\caption{Results for the chinese digits dataset.}
\label{fig:chinese_results}
\end{figure}

\begin{figure}
\centering
\includegraphics[scale=0.35]{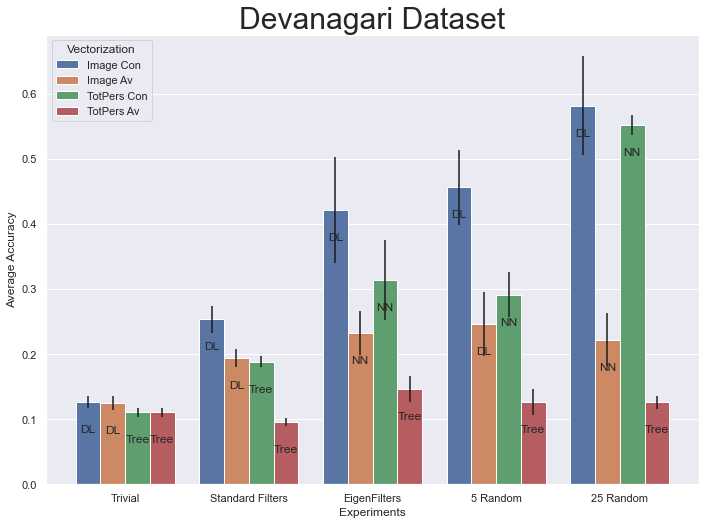}
\caption{Results for the Devanagari dataset.}
\label{fig:dev_results}
\end{figure}

\begin{figure}
\centering
\includegraphics[scale=0.35]{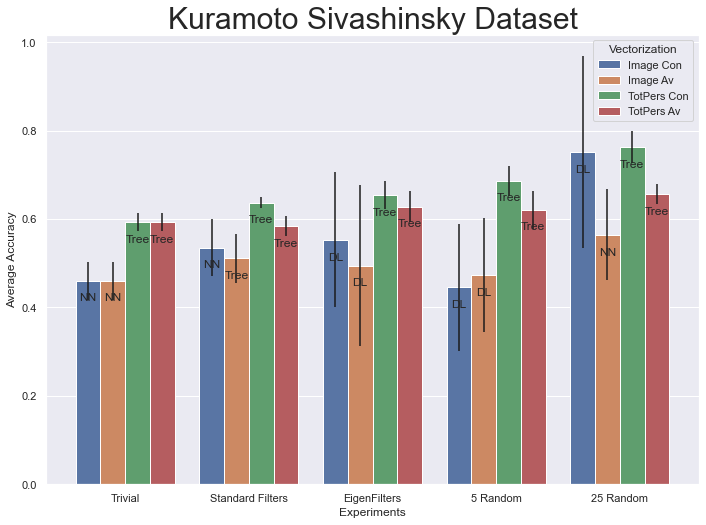}
\caption{Results for the Kuramoto-Sivashinsky dataset.}
\label{fig:KS_results}
\end{figure}

%

\section{Conclusion}
\label{sec:conclusion}
In this paper, we have proposed a novel approach to topological analysis of images and simplicial complexes, based on combining convolutions and persistent homology. Packaging this invariant as a topological transform, we have shown that this transform is injective in a generic sense by viewing it as a special case of the persistent homology transform. We have also provided diverse experiments showing that convolutions are effective at improving the predictive power of persistent homology for classification tasks. Looking forward, there are multiple directions for future research.

\begin{itemize}
\item The injectivity theory for the PHT has been extended by results bounding the number of vector directions needed to obtain injectivity for certain classes of shapes, see \citep{curry2018many,belton2018learning,fasy2019persistence,belton2020reconstructing}. It would be interesting to try and obtain similar, specialized results for the embedded cubical complexes appearing in this paper.
\item As of yet, there are no \emph{inverse stability} results for the PHT, as were obtained for the topological transform in \citep{solomon2021geometry} . That is, the regularity of the inverse map from topological features to shapes has not been investigated. It is possible that the restricted, discrete structure of the CPT provides an easier setting for studying this question.
\item Both versions of convolutional persistence explored in this paper are discrete. It would be interesting to study convolutional persistence in the continuous setting. The major obstacle to obtaining an inverse result in this setting is that the inverse theory for the persistent homology transform, as derived in  \citep{turner2014persistent,curry2018many,ghrist2018persistent}, is that one must work within an O-minimal structure of definable sets, so that one must prove that all Euclidean embeddings used are definable in a precise sense. This is immediate when working with discrete objects but presents a novel technical challenge otherwise, and so falls outside of the scope of this paper.
\item There are many choices to be made in the implementation of convolutional persistence, and it is not clear which give the best result. For example, whether using a big stride and many filters or a small stride and few filters is preferable. Another experiment is to consider whether it is best to work with eigenfilters directly, rather than random normalized linear combinations thereof, as we have done here.
\item  We did not consider a pipeline in which the convolutional filters are trainable parameters learned using the dataset. This is entirely possible, see \citep{carriere2021optimizing,solomon2021fast}, but presents multiple computational challenge, as topological optimization is expensive and unstable, and the topological energy landscape has many poor local minima. Still, this is an interesting direction for future work.
\item We considered persistence images and total persistence as vectorizations in our experimental pipeline. Other possible vectorizations include persistence landscapes \citep{bubenik2015statistical} and Euler characteristic curves.
\item The datasets explored here, though not synthetic, are fairly simplistic, and such classification tasks have been \emph{solved} by deep learning methods which surpasses human accuracy \citep{russakovsky2015imagenet}. A natural next step is to consider more challenging settings, particularly in the data-starved regime, where topological features have proved useful \citep{khramtsova2022rethinking} for obtaining start-of-the-art performance, and see if convolutional persistence can further boost the accuracy of such models. 
\item Finally, it would be interesting to experiment with multi-channel images (e.g. RGB$\alpha$), 3D images (as often arise in medical imaging), as well as datasets of simplicial complexes (like graphs, meshes).
\end{itemize}

\bibliography{sn-bibliography}


\end{document}